\documentclass[12pt]{amsart}

\usepackage{xcolor}
\usepackage[normalem]{ulem} 

\newtheorem{thm}{Theorem}[section]
\newtheorem{cor}[thm]{Corollary}

\newtheorem{prop}[thm]{Proposition}
\theoremstyle{definition}
\newtheorem{defn}[thm]{Definition}
\newtheorem{rem}[thm]{Remark}

\newtheorem{ex}[thm]{Example}
\numberwithin{equation}{section}

\theoremstyle{plain}

\usepackage{amsmath}
\usepackage{amsfonts}
\usepackage{amssymb}
\usepackage{color}
\usepackage{mathtools}
\usepackage{bbm}
\newcommand{\be}{\begin{equation}}
	\newcommand{\en}{\end{equation}}
\newcommand{\Lc}{{\mc L}}
\newcommand{\bei}{\begin{itemize}}
	\newcommand{\eni}{\end{itemize}}
\newcommand{\ip}[2]{\langle{#1}|{#2}\rangle}

\newcommand{\C}{\mathfrak{C}}

\numberwithin{equation}{section}

\newcommand{\mb}{\mathbb}
\newcommand{\mc}{\mathcal}
\newcommand{\eul}{\mathfrak}
\newcommand{\A}{\eul A}
\newcommand{\Ao}{{\eul A}_{0}}

\newcommand{\D}{\mc D}

\newcommand{\IA}{{\mathcal I}_{\Ao}^{\,\YY}(\A)}

\newcommand{\id}{{\sf e }}

\newcommand{\YY}{\mathfrak Y}

\newcommand{\KK}{\mathfrak K}

\newcommand{\B}{{\eul B}}

\newcommand{\ad}{^{\mbox{\scriptsize $\dag$}}}

\newcommand{\X}{{\mathfrak X}}
\newcommand{\Z}{{\mathfrak Z}}

\newcommand{\mult}{\,{\scriptstyle \square}\,}
\newcommand{\vp}{\Phi}

\newcommand{\LDK}{{\mathcal L}\ad(\D,\mathcal{K})}

\newcommand{\idop}{{\mb I}}

\def\MM{{\mathfrak M}}

\numberwithin{equation}{section}

\newcommand{\ppi}{\Pi}
\def\H{\mc H}
\def\x{\relax\ifmmode {\mbox{*}}\else*\fi}

\newcommand{\up}{\raisebox{0.7mm}{$\upharpoonright$}}

\newcommand{\vertiii}[1]{{\left\vert\kern-0.25ex\left\vert\kern-0.25ex\left|#1 
					\right\vert\kern-0.25ex\right\vert\kern-0.25ex\right\vert}}
                    
\definecolor{Magenta}{rgb}{1,0,1}
\definecolor{ao(english)}{rgb}{1,0,0.4}

\newcommand{\BibTeX}{B\kern-0.1emi\kern-0.017emb\kern-0.15em\TeX}
\newcommand{\XYpic}{$\mathrm{X\kern-0.3em\raisebox{-0.18em}{Y}}$-$\mathrm{pic}\,$}

\newcommand{\ed}{\end{document}}
\begin{document}

%
%
%
%
%
%
%
%
%

	\title[]{BANACH BIMODULE-VALUED POSITIVE MAPS: INEQUALITIES AND INDUCED REPRESENTATIONS}

	\author[G. ~Bellomonte]{Giorgia Bellomonte}

%
\address{Dipartimento di Matematica e Informatica, Universit\`a degli Studi  di Palermo, Via Archirafi n. 34,  I-90123 Palermo, Italy}
\email{giorgia.bellomonte@unipa.it} 

%

\author[S. ~Ivkovi\'{c}]{Stefan Ivkovi\'{c}}

\address{Mathematical Institute of the Serbian Academy of Sciences and Arts, Kneza Mihaila 36, 11000 Beograd, Serbia}
\email{stefan.iv10@outlook.com}

\author[C. ~Trapani]{Camillo Trapani}
\address{Dipartimento di Matematica e Informatica, Universit\`a degli Studi  di Palermo, Via Archirafi n. 34,  I-90123 Palermo, Italy}
\email{camillo.trapani@unipa.it} 

\subjclass{Primary 46K10, 47A07, 16D10. }
\keywords{Representations, modules,
positive sesquilinear  maps, 
completely positive sesquilinear  maps, 
  normed spaces.
}
\date{\today}
\begin{abstract} Representations induced by  general positive sesquilinear maps with values in ordered Banach bimodules such as commutative and non-commutative $L^1$-spaces and the spaces of bounded linear operators from a $C^*$-algebra into the dual of  another $C^*$-algebra are considered.  As a starting point, a generalized Cauchy-Schwarz inequality is proved for these maps  and a representation of bounded positive maps from a (quasi) *-algebra into  such an ordered Banach bimodule is derived and some more inequalities for these maps are deduced. 
In particular,  an extension of Paulsen's modified Kadison-Schwarz inequality for 2-positive maps to the case of general positive maps from a unital *-algebra into the space of trace-class operators on a separable Hilbert space and into the duals of von-Neumann algebras is obtained.  Also, representations for completely positive maps with values in an ordered Banach bimodule and Cauchy-Schwarz inequality for infinite sums of such maps are provided. Concrete examples illustrate the results.
\end{abstract}
\label{page:firstblob}
\maketitle
 \section{Introduction} 

 Positive and completely positive maps play an important role in the theory of operator algebras  and quantum information, see for instance \cite{DKM}. 
 This motivate a study of  representations involving these maps. 
 A classical result in representation theory  is the famous Stinespring theorem \cite{stinespring} (see also \cite{Arveson,paulsen,stormer} which gives a representation of completely positive maps from a $C^*$-algebra into the space of bounded operators in Hilbert space. Quite recently, in \cite[Corollary 3.10]{BIvT1} and \cite[Corollary 3.3]{BDIv},  a representation of general positive $C^*$-valued maps on arbitrary *-algebras has also been provided. Now, $C^*$-algebras are just a special case of Banach quasi *-algebras and, moreover, every Banach quasi *-algebra is a Banach bimodule over a *-algebra. 
  Several beautiful and important mathematical structures (such as commutative and noncommutative $L^p$-spaces) are examples of Banach quasi *-algebras. Further, Hilbert $C^*$-modules \cite{Lance, MT} and the spaces of  bounded linear operators from a $C^*$-algebra into the dual of  another $C^*$-algebra are other examples of (ordered) Banach bimodules over a *-algebra which are of interest for applications.   All these facts motivated us to study positive and completely positive maps from a Banach quasi *-algebra into an (ordered) Banach bimodule over a *-algebra, and the main aim of this paper is therefore obtaining representations and related inequalities of such maps.
In a recent paper \cite{BIvT1}, we considered a GNS-like construction defined by positive sesquilinear maps on a quasi *-algebra taking their values in a $C^*$-algebra. In this paper we go some steps further and consider the possibility of replacing sesquilinear maps taking values in a $C^*$-algebra with sesquilinear maps with values in an ordered Banach bimodule $\YY$ over some *-algebra $\YY_0$. In this situation, an extension of the result of \cite{BIvT1} is possible, provided that the sesquilinear form $\Phi$, where we start from, satisfies a generalized Cauchy-Schwarz inequality. This discussion is developed in Proposition \ref{prop: 3.6 cases} in Section \ref{sect_3} where we characterize a class of ordered Banach bimodules $\YY$ with the property that every $\YY$-valued positive sesquilinear map satisfies a generalized Cauchy-Schwarz inequality.    
Concrete examples of such bimodules are, for instance, non-commutative $L^1$-spaces and the space of bounded linear operators from a von Neumann algebra into the dual of another von Neumann algebra. As a corollary, we  obtain an extension of Paulsen's modified Kadison-Schwarz inequality for 2-positive maps between $C^*$-algebras to the case of general positive maps from a unital *-algebra into any of the above mentioned ordered Banach bimodules (Corollary \ref{cor: positive boundd l}) as well as Cauchy-Schwarz inequality for infinite sums of ordered  Banach bimodule-valued positive sesquilinear maps (Proposition \ref{prop: 3.16 sequences}). Then we provide a representation of positive sesquilinear maps from a quasi *-algebra into such an ordered Banach bimodule. Moreover, we give examples of such maps. In the last part of Section \ref{sect_3}, we obtain also a representation of positive ordered Banach bimodule-valued linear maps on a unital *-algebra (Corollary \ref{cor: positive boundd l}).  An important difference between representations of ordered Banach bimodule-valued positive maps given in Section \ref{sect_3}  and the representations of $C^*$-valued positive maps given in \cite{BIvT1} is that the ones from  Section \ref{sect_3}  are given in a proper Banach space and not just in a quasi-Banach space (as in \cite{BIvT1}). This is due to the Cauchy-Schwarz inequality for positive   Banach bimodule-valued sesquilinear maps given in Proposition \ref{prop: 3.6 cases}. 
Positive maps from a unital $C^*$-algebra into the space of trace-class operators are a special case of the positive maps treated in Section \ref{sect_3}. Since the space of trace-class operators is the subspace of the $C^*$-algebra of bounded linear operators on a Hilbert space, it follows that our results have an application in the theory of positive maps on $C^*$-algebras.
 Section \ref{Sect_4} is devoted to the *-representations induced by completely positive  sesquilinear maps.
 This is actually a subject for which several approaches  have been proposed  in view of a possible generalization of the Stinespring theorem to different environments (see, for instance, \cite{Asad,Joita, heo, KK, PY, GMX, MJJ}).
 Motivated by \cite[Proposition 3.1]{PY}, following an idea proposed in \cite[Ch. 11]{schm_book} and adopted in \cite{BIT} for partial *-algebras, we consider completely positive sesquilinear maps from a normed quasi *-algebra $\A$ into a set of sesquilinear maps on a vector space which take values in an ordered Banach bimodule $\YY$ and provide a generalization  of the classical Stinespring theorem  in this context, as well as applications of this result to Cauchy-Schwarz inequality for infinite sums of such maps (Corollary \ref{cor: infinite sums}). Further, we provide a version of Radon-Nikodym theorem for these maps  as well as for   ordered Banach bimodule-valued bounded positive sesquilinear maps on a quasi *-algebra, thus generalizing, in this setting, the results from \cite[Section 3.5]{stormer}  and give  examples of such maps. At the end of the section, we apply our results to obtain representations of  bounded completely positive sesquilinear maps on a quasi *-algebra with values in the space of ordered Banach bimodule-valued,  bounded, linear operators.  As a consequence of these representations, in Proposition \ref{prop: von Neumann} we obtain Cauchy-Schwarz inequalities for such maps. Also, we provide some examples as well as some other applications.

 \section{Preliminaries} 

A {\em quasi *-algebra} $(\A, \A_0)$ is a pair consisting of a vector space $\A$ and a *-algebra $\A_0$ contained in $\A$ as a subspace and such that
\begin{itemize}
	\item $\A$ carries an involution $a\mapsto a^*$ extending the involution of $\A_0$;
	\item $\A$ is  a bimodule over $\A_0$ and the module multiplications extend the multiplication of $\A_0$. In particular, the following associative laws hold:
	\begin{equation}\notag \label{eq_associativity}
		(ca)d = c(ad); \ \ a(cd)= (ac)d, \quad \forall \ a \in \A, \  c,d \in \A_0;
	\end{equation}
	\item $(ac)^*=c^*a^*$, for every $a \in \A$ and $c \in \A_0$.
\end{itemize}

The
\emph{identity} or {\it unit element} of $(\A, \A_0)$, if any, is a necessarily unique element $\id\in \A_0$, such that
$a\id=a=\id a$, for all $a \in \A$.

We will always suppose that
\begin{align*}
	&ac=0, \; \forall c\in \A_0 \Rightarrow a=0 \\
	&ac=0, \; \forall a\in \A \Rightarrow c=0. 
\end{align*}
Clearly, both these conditions are automatically satisfied if $(\A, \A_0)$ has an identity $\id$.\\

A  quasi *-algebra $(\A, \A_0)$ is said to be  {\em normed} if $\A$ is a normed space, with a norm $\|\cdot\|$ enjoying the following properties
\begin{itemize}
	\item there exists $\gamma>0$ such that for every $a\in\A$ $$ 
\max\{\|ac\|, \|ca\|\}\leq \gamma \|a\|, \quad\forall c\in \A_0;$$
	\item $\|a^*\|=\|a\|, \; \forall a \in \A;$
	\item  $\A_0$ is dense in $\A[\|\cdot\|]$.
\end{itemize}
If the normed vector space $\A[\|\cdot\|]$ is complete, then $(\A, \A_0)$ is called a  {\em Banach  quasi *-algebra}. We refer to \cite{FT_book} for further details.

\begin{defn}
Let  $\YY$ be a Banach bimodule  over the  *-algebra $\YY_0$. 
We say that $\YY$ is  an {\em ordered Banach bimodule} over $\YY_0$ if
\begin{itemize}
\item[(i)]  $\YY$ is ordered as a vector space; that is,
$\YY$ contains  a (positive) closed cone $\KK$, i.e., $\KK\subset \YY$ is such that $\KK+\KK\subset \KK$, $\lambda\KK \subset \KK $ for $\lambda\geq 0$ and $\KK\cap (-\KK)=\{0\}$; 
\item[(ii)] $z^* \KK z \subset \KK, \quad \forall z\in \YY_0.$
\end{itemize} 
\end{defn}

    As usual, we will write $y_1\leq y_2$ whenever $y_2-y_1\in\KK$, with $y_1,y_2\in\YY$ and we will sometimes suppose that $\YY$ has an order-preserving norm in the sense that if $y_1\leq y_2$ with $y_1,y_2\in\YY$, then also $\|y_1\|_\YY\leq \|y_2\|_\YY$. 
Throughout the paper, $\Omega$ will denote a (locally) compact Hausdorff space and $C(\Omega)$ and $M(\Omega)$ will denote the space of continuous functions on $\Omega$ and  the Banach space of all complex Radon measures
on $\Omega$ equipped with the total variation norm, respectively. 
\begin{ex}\label{ex: submultiplicative}
 Examples of ordered Banach bimodules over a Banach *-algebra are provided by: \begin{itemize}\item if  $\H$ is a separable Hilbert space, and $\B_1(\H)$
denotes the space of all trace-class operators, then $\YY=(\B_1(\H), \|\cdot\|_1)$ is a Banach bimodule over the von Neumann algebra  $\YY_0=(\B(\H),\|\cdot\|)$  of bounded operators on $\H$;  
         \item the non commutative spaces $\YY= L^p(\rho)$,  over the  Banach *-algebra  $\YY_0= L^\infty(\rho)$, with $\rho$ a faithful finite trace on a von Neumann algebra $\MM$, see \cite[Section 5.6.1]{FT_book};
                  \item $\YY= \B(\C_1,\C_2)$, with $\C_1$ and $\C_2$ being $C^*$-algebras with unit, $\YY_0=\C_1$.   
        If $A\in\B(\C_1,\C_2)$ and $c\in\C_1$, the module multiplication can be defined as $$(A\cdot c)(d)=A(cd) \mbox{ and } (c\cdot A)(d)=A(dc),\quad  d\in\C_1.$$
        Similar construction applies if we instead of $\C_2$ consider the dual of $\C_2$;
 \item  $\YY=\YY_0= \ell_2(C(\Omega))$,  with $\ell_2(C(\Omega))$  the space of squared summable sequences of functions in $C(\Omega)$ with  $\Omega$ a compact Hausdorff space. It is a Banach bimodule over itself because it is a (non-unital) normed *-algebra.
     \end{itemize}  
 The definitions of the positive cones are obvious in all these cases.
 \end{ex}

  In the next results  we will need the additional assumption that the norm on the ordered Banach bimodule $\YY$ is order-preserving.
Examples of Banach bimodules with an
order-preserving norm are commutative $L^p$-spaces and $\ell_2(C(\Omega))$,  with $\Omega$  a compact Hausdorff space.  Furthermore, if $\C_1$ and $\C_2$ are $C^*$-algebras with unit $e_{\C_1}$ and $e_{\C_2}$, respectively, the Banach ordered bimodule $\B(\C_1,\C_2)$ is order-preserving (this follows from the fact that $\|\varphi\|=\|\varphi(e_{\C_1})\|_{\C_2}\leq\|\psi(e_{\C_1})\|_{\C_2}=\|\psi\|$, with $\varphi,\psi\in \B(\C_1,\C_2)^+$ and $\varphi\leq\psi$, see e.g. \cite{stormer}). Moreover, if $\H$ is a separable Hilbert space, the norm in $\B_1(\H)$ is clearly order-preserving.


Let $\X$ be a vector space and $\YY$ an ordered Banach bimodule over the *-algebra $\YY_0$, with positive cone $\KK$. Let $\vp$ be a $\YY$-valued positive  sesquilinear  map on   $\X\times\X$  $$\vp:(x_1,x_2)\in\X\times\X\to\vp(x_1,x_2)\in\YY$$ i.e., a map with the properties  \begin{itemize}
	\item[$i)$] $\vp(x_1,x_1)\in \KK$,
	\item[$ii)$]$\vp(\alpha x_1+\beta x_2,\gamma x_3)=\overline{\gamma}[\alpha\vp( x_1,x_3)+\beta \vp(x_2,x_3)]$,  
\end{itemize}
with $x_1, x_2,x_3 \in\X$ and $\alpha,\beta,\gamma\in\mathbb{C}$. \\
The $\YY$-valued positive  sesquilinear  map $\vp$ is called {\em faithful} if $$\vp(x,x)=0_\YY \;\Rightarrow\; x=0.$$

\section{Representations induced by positive  maps}\label{sect_3}

It is quite apparent that the cornerstone of every extension of the GNS-construction is a Cauchy-Schwarz inequality. We begin with presenting some results in this direction.\\

	Throughout this section we will assume that  $\C$ is a $C^*$-algebra with positive convex cone $\C^+$ and norm $\|\cdot\|_\C$, $ \YY$ is (an ordered) Banach bimodule over a  *-algebra $\YY_0$  (with $\YY_0$ equipped with a  not necessarily sub-multiplicative norm $\|\cdot\|_{\YY_0}$) with the respective cones $\KK$ and $\KK_0$, with $\KK_0=\{\sum_{i=1}^Nz_i^*z_i,\, z_i\in\YY_0, i=1, \dots, N;  N\in\mathbb{N}\}$.\\

 From now on, by $\rho$ we will denote a  faithful semifinite trace on a von Neumann algebra $\MM$.
\begin{prop}\label{prop: 3.6 cases}	   Let $\X$ be a complex vector space. Assume that $\YY$, 
    $\mathfrak{Y}_0$, 
    $\KK$ and $\C$ are as above. Let $\Phi : \X \times \X \to \YY$ be a positive  sesquilinear map. Let $\Omega, \Upsilon$ be  locally compact Hausdorff spaces,  $\MM$ a von Neumann algebra and  $\rho$ a faithful trace on  $\MM$.  Then, the Cauchy-Schwarz inequality holds in each one of the following cases  \begin{enumerate}
        \item  $\YY=L^2(\Omega)$,
         \item $\YY=\B(\MM, \B_1(\H))$, $\YY_0=\MM$,
         
        \item $\YY=\MM^*$, the dual of $\MM$,
                \item $\YY= M(\Omega)$ and $\YY_0= C(\Omega)$, \item $\YY=\B(\MM, M(\Omega))$, $ \YY_0=\MM$,
        \item $\YY=\B(C(\Omega),C(\Upsilon))$,
        or $\YY=\B(\MM, C(\Omega))$,
        \item $\YY=L^1(\rho)$ or $\YY=L^1(\Omega)$,
        \item  $\YY=\B(\MM,\widetilde{\MM}^*)$ and $\YY_0=\MM$, with $\widetilde{\MM}$  another von Neumann algebra,
        \item $\YY=\YY_0= \ell_2(C(\Omega))$.
    \end{enumerate} 
	\end{prop}\begin{proof}
	    	We just give a sketch of the proof of each case.\\
	 (1) Let $S$ be a simple function on $\Omega$, then $S=\sum_{j=1}^Nc_j\chi_{A_j}.$ Then, for each $j$, the map $\varphi_j(x_1,x_2)=\int_\Omega \Phi(x_1,x_2)\chi_{A_j}d\mu$, $x_1,x_2\in\X$, is a positive  sesquilinear form, so we can apply the Cauchy-Schwarz inequality and get:
    $$\left|\int_\Omega \Phi(x_1,x_2)Sd\mu\right|
    \leq \|\Phi(x_1,x_1)\|_2^{1/2}\|\Phi(x_2,x_2)\|_2^{1/2}\|S\|_2, \quad\forall x_1,x_2\in\X;$$\\
    Then consider a sequence $\{S_n\}_n$ of simple functions such that $S_n\to\overline{\Phi(x_1,x_2)}$ in $L^2(\Omega,\mu)$ as $n\to\infty$. By taking limits in the above inequality, we get the Cauchy-Schwarz inequality.
     \\(2) Let $P\in\MM$ and $Q\in \B(\H)$ be orthogonal projections. Define $\varphi_{P,Q}:\X\times\X \to\mathbb{C}$ by \begin{equation}
         \label{eq: phiPQ}\varphi_{P,Q}(x_1,x_2)=tr(Q(\Phi(x_1,x_2)(P))),\quad \forall x_1,x_2\in\X.
     \end{equation} Then $\varphi_{P,Q}$ is a positive sesquilinear form; hence, using the Cauchy-Schwarz inequality, we get
 \begin{multline}\label{eq: trace}     
   |tr(Q(\Phi(x_1,x_2)(P)))|\\\leq(tr(Q(\Phi(x_1,x_1)(P))))^{1/2}(tr(Q(\Phi(x_2,x_2)(P))))^{1/2}
\end{multline}for all $x_1,x_2\in\X$. Let now $S\in \MM$ and $T\in \B(\H)$ be two simple operator functions, so $S=\sum_{i=1}^N\alpha_iP_i$
and  $T=\sum_{j=1}^M\beta_j Q_j$ with $P_{i_1}P_{i_2}=Q_{j_1}Q_{j_2}=0$ whenever $i_1\neq i_2$ and $j_1\neq j_2$,   $\sum_{i=1}^N P_i={\idop_{\MM}}$ and $\sum_{j=1}^M Q_j={\idop_{\H}}$. By \eqref{eq: trace} and some computations we get that      \begin{align}\label{eq: ineq*}
     |&tr(T(\Phi(x_1,x_2)(S)))|\leq\sum_{i,j}|\alpha_i||\beta_j||tr(Q_j(\Phi(x_1,x_2)(P_i)))|\\&\nonumber\qquad\leq \|T\|\|S\| (tr(\Phi(x_1,x_1)(\idop_{\MM})))^{1/2}(tr(\Phi(x_2,x_2)(\idop_{\MM})))^{1/2}\\&\nonumber\qquad\leq\|T\|\|S\|\|\Phi(x_1,x_1)\|^{1/2}\|\Phi(x_2,x_2)\|^{1/2},
\quad \forall x_1,x_2\in\X.
\end{align} Let now $U\in\MM$ and $W\in\B(\H)$ be unitary. Then, by the spectral theorem, we can find sequences $\{S_n\}_n\subset\MM$ and $\{T_n\}_n\subset\B(\H)$ that converge in the operator norm to $U$ and $W$, respectively.    

Then, since \begin{align}\label{eq: once again}
     |&tr(B(\Phi(x_1,x_2)(A)))|\leq tr(|B(\Phi(x_1,x_2)(A))|)\\&\nonumber\qquad\leq \|B\| \|\Phi(x_1,x_2)(A)\|_1\leq  \|B\|\|A\| \|\Phi(x_1,x_2)\|,
\end{align} for all $A\in \MM$, $B \in \B(\H)$, $x_1,x_2\in\X$, it follows that \begin{equation*}
     tr(T_n(\Phi(x_1,x_2)(S_n)))\to tr(W(\Phi(x_1,x_2)(U))),\, \mbox{ as }n\to\infty.
\end{equation*}  Also, $\|T_n\|\|S_n\|\to \|W\|\|U\|=1$ as $n\to\infty$, hence, by using  inequality \eqref{eq: ineq*}, passing to the limits, we deduce that \begin{equation}\label{eq: another ineq}
     |tr(W(\Phi(x_1,x_2)(U)))|\leq\|\Phi(x_1,x_1)\|^{1/2}\|\Phi(x_2,x_2)\|^{1/2},
\end{equation}  for all $x_1,x_2\in\X$. This inequality extends easily to convex combinations $\sum_{i=1}^N\lambda_iU_i$, $\sum_{j=1}^M\gamma_jW_j$. 

\noindent The Russo-Dye theorem then allows to get the following inequality
for $A\in\MM$ and $B\in\B(\H)$ with $\|A\|,\|B\|<1$, 
\begin{multline}
    \label{eq: I cba}
|tr(B(\Phi(x_1,x_2)(A)))|\\\leq\|\Phi(x_1,x_1)\|^{1/2}\|\Phi(x_2,x_2)\|^{1/2}, \quad \forall  x_1,x_2\in\X.\end{multline} Finally, let $A\in\MM$ with $\|A\|=1$, and $V\in\B(\H)$ be the partial isometry with $ V(\Phi(x_1,x_2)(A))=|\Phi(x_1,x_2)(A)|$ and choose sequences $\{A_n\}_n\subset\MM$ and $\{V_n\}_n\subset\B(\H)$, with $\|A_n\|,\|V_n\|<1$ for all $n$, such that $A_n\to A$ and $V_n\to V$ in the operator norm as $n\to\infty$. Once again, by using \eqref{eq: once again} and a limit procedure we get 
\begin{equation*}
   \|\Phi(x_1,x_2)(A)\|_1=   tr(V(\Phi(x_1,x_2)(A)))\leq\|\Phi(x_1,x_1)\|^{1/2}\|\Phi(x_2,x_2)\|^{1/2},
\end{equation*} for all  $x_1,x_2\in\X$. Finally,  by taking the supremum over the unit ball in $\MM$, we conclude that \begin{equation*}
\|\Phi(x_1,x_2)\|\leq\|\Phi(x_1,x_1)\|^{1/2}\|\Phi(x_2,x_2)\|^{1/2},
\quad \forall x_1,x_2\in\X.\end{equation*}
\noindent (3)
It is a special case of $(2)$: if $\H$ is one dimensional, then $\B_1(\H)=\mathbb{C}$.

    \noindent (4)      We can apply the Cauchy-Schwarz inequality to every positive sesquilinear form $\varphi_E(x_1,x_2)=\Phi(x_1,x_2)(E)$, with $E$ a given Borel subset of $\Omega$, and therefore by  considering a   partition $\{E_n\}_{n=1}^N$   of $\Omega$  we obtain \begin{align*}
        \sum_{n=1}^N|\Phi(x_1,x_2)(E_n)|&\leq  \sum_{n=1}^N(\Phi(x_1,x_1)(E_n))^{1/2}(\Phi(x_2,x_2)(E_n))^{1/2}\\&\leq\left(\sum_{n=1}^N\Phi(x_1,x_1)(E_n)\right)^{\frac{1}{2}}\left(\sum_{n=1}^N\Phi(x_2,x_2)(E_n)\right)^{\frac{1}{2}}
	\end{align*} hence, by taking the supremum of both sides, we get the Cauchy-Schwarz inequality.\\ \noindent (5)
Let $P\in\MM$ be an orthogonal projection and $E\subset$ be a Borel subset of $\Omega$. The  map $\varphi_{P,E}:\X\times \X\to \mathbb{C}$ given by $$\varphi_{P,E}(x_1,x_2)=\Phi(x_1,x_2)(P)(E), \quad\forall x_1,x_2\in\X$$ is a positive sesquilinear form, hence, by similar arguments as in the proof of $ (4)$ we can deduce that $$\|\Phi(x_1,x_2)(P)\|\leq(\Phi(x_1,x_1)(P))^{1/2}(\Phi(x_2,x_2)(P))^{1/2}, \quad\forall x_1,x_2\in\X.$$ Since this holds for an arbitrary orthogonal projection $P\in\MM$, we can then proceed further in a similar way as in the proof of $(2)$ and deduce Cauchy-Schwarz inequality for $\Phi$. We leave the details to the reader. \\
	         (6)   The map $\widetilde{\varphi}_t:\X\times \X\to (C(\Omega))^*\cong M(\Omega)$ given by$$\widetilde{\varphi}_t(x_1,x_2)(f)=(\Phi(x_1,x_2)(f))(t), \quad \forall f\in C(\Omega), x_1,x_2\in\X$$ is sesquilinear and positive, hence, the Cauchy-Schwarz inequality holds for $\widetilde{\varphi}_t$ by case (4). Therefore, for every $f\in C(\Omega)$, with $\|f\|\leq1$ and every $t\in\Upsilon$, we get \begin{align*}
                 &\left|(\Phi(x_1,x_2)(f))(t)\right|\leq\|\widetilde{\varphi}_t(x_1,x_2)\|\leq\|\widetilde{\varphi}_t(x_1,x_1)\|^{1/2}\|\widetilde{\varphi}_t(x_2,x_2)\|^{1/2}\\&\qquad\quad\leq
             \left(\sup_{\substack{g\in C(\Omega), \\\|g\|_\infty\leq1}}\|\Phi(x_1,x_1)(g)\|_\infty\right)^{1/2}\left(\sup_{\substack{h\in C(\Omega), \\\|h\|_\infty\leq1}}\|\Phi(x_2,x_2)(g)\|_\infty\right)^{1/2}\\&\qquad\quad\leq \|\Phi(x_1,x_1)\|^{1/2}\|\Phi(x_2,x_2)\|^{1/2}.\end{align*}  The statement follows by taking supremums.\\ The case $\YY=\B(\MM, C(\Omega))$ can be treated similarly by using (3).\\
               (7) 
               Consider the map $\iota:L^1(\rho)\to\widetilde{\MM}^*$ given by $$\iota(T)(S)=\rho(ST),\quad\forall \,T\in L^1(\rho), S\in\MM.$$    
              It is not hard to check that $\iota$ is an isometry. Moreover, since $\iota(T)(S)=\rho(S^{1/2}TS^{1/2})\geq0$ whenever $S$ and $T$ are positive, it follows that $\iota$ preserves positivity. Hence, if $\Phi:\X\times\X\to L^1(\rho)$ is a positive sesquilinear map, then $\iota\circ\Phi:\X\times\X\to\MM^*$ is also a positive sesquilinear map. By (3) and since $\iota$ is an isometry, we get that \begin{align*}
\|\Phi(x_1,x_2)\|_1&=\|\iota\circ\Phi(x_1,x_2)\|\leq  \|\iota\circ\Phi(x_1,x_1)\|^{1/2}\|\iota\circ\Phi(x_2,x_2)\|^{1/2}\\&=\|\Phi(x_1,x_2)\|_1^{1/2}\|\Phi(x_1,x_2)\|_1^{1/2},\quad \forall x_1,x_2\in\X.
              \end{align*}
Of course, an analogous result holds in the commutative case (i.e. $\YY=L^1(\Omega)$).\\
(8) In the proof of $(2)$, given $P\in \MM$ and $Q\in \B(\H)$, we have considered the sesquilinear form $\varphi_{P,Q}$ in \eqref{eq: phiPQ}. Now, given $S\in \MM$ and $T\in\widetilde{\MM}$, we can consider, instead, the sesquilinear form $\varphi_{S,T}:\X\times \X\to \mathbb{C}$ given by $$\varphi_{S,T}(x_1,x_2)=\Phi(x_1,x_2)(S)(T),\quad\forall x_1,x_2\in\X$$
and proceed analogously to the proof of $(2)$.
\\
(9) We can apply the Cauchy-Schwarz inequality to all the positive sesquilinear forms  $\varphi_{n,t}(x_1,x_2)=(\Phi(x_1,x_2))_n(t)$ with $(\Phi(x_1,x_2))_n$ the n-th component of the sequence $\Phi(x_1,x_2)$ and then applying the Cauchy-Schwarz inequality for the inner product in  $\ell_2(C(\Omega))$.
  \end{proof}
From now on $\YY$ will denote any Banach bimodule considered in Proposition \ref{prop: 3.6 cases}.
  The next corollary is motivated by modified Schwarz inequality for 2-positive maps (see \cite[Ch. 3, Ex. 3.4, p.40]{paulsen} and \cite[Lemma 2.6]{bhatt})
   \begin{cor}\label{cor: positive boundd l}
     If $\A$ is a unital *-algebra and $\omega:\A\to\YY$ is a positive linear map, then $$\|\omega(b^*a)\|_\YY\leq\|\omega(a^*a)\|_\YY^{1/2}\|\omega(b^*b)\|_\YY^{1/2},\quad\forall a,b\in \A.$$\end{cor}\begin{proof}
         This follows from Proposition \ref{prop: 3.6 cases} by considering the  sesquilinear map
  $\Phi:\A\times \A\to\YY$ given by $\Phi(a,b)=\omega(b^*a)$, for all $a,b\in\A$.
     \end{proof} 
  Now we give an extension of Proposition \ref{prop: 3.6 cases}.
   \begin{prop}\label{prop: 3.16 sequences}
      Let $\X$ be a vector space and assume that the norm on $\YY$ is order-preserving. If $\{\Phi_n\}_n$ is a sequence of positive sesquilinear maps $\Phi_n:\X\times\X\to\YY$ and $\{x_n\}_n$, $\{\widetilde{x}_n\}_n$ are sequences in $\X$ such that the series $\sum_{n=1}^\infty \Phi_n(x_n,x_n)$ and $\sum_{n=1}^\infty \Phi_n(\widetilde{x}_n,\widetilde{x}_n)$ are convergent in $\YY$ then $$
\left\|\sum_{n=1}^\infty \Phi_n(x_n,\widetilde{x}_n)\right\|_\YY\leq \left\|\sum_{n=1}^\infty \Phi_n(x_n,x_n)\right\|_\YY^{1/2}\left\|\sum_{n=1}^\infty \Phi_n(\widetilde{x}_n,\widetilde{x}_n)\right\|_\YY^{1/2}.
$$
  \end{prop}\begin{proof}
      The statement can be proved by arguments similar to the first part of \cite[Example 1.3.5]{MT} due to Proposition \ref{prop: 3.6 cases} and the fact that the norm of $\YY$ is order-preserving.
  \end{proof}

For a detailed overview of some other operator-inequalities for positive maps we refer to \cite{BKMS,KMN}.




\begin{defn}\label{defn: $C^*$-valued quasi inner product}
Let $\X$ be a vector space. A $\YY$-valued   faithful positive  sesquilinear map $\vp$ on   $\X\times\X$ is said to be a {\em $\YY$-valued  inner product}  on $\X$ and we often  write $\ip{x_1}{x_2}_\vp:=\vp(x_1,x_2)$,  $x_1,x_2\in\X$.
\end{defn}

 A $\YY$-valued   inner product  on $\X$ $$\vp:\X\times\X\to \YY$$    induces a   norm $\|\cdot\|_\vp$ on $\X$:

\begin{equation*}
\|x\|_\vp:=\sqrt{\|\ip{x}{x}_\vp\|_\YY}=\sqrt{\|\vp(x,x)\|_\YY},\quad x\in\X,\end{equation*} 
since
 \begin{align*} 	
&   \|x_1+x_2\|_\vp\leq \| x_1\|_\vp+\| x_2\|_\vp, \qquad  \forall x_1,x_2\in\X;
\end{align*}  this can be shown similarly to what done in \cite{BIvT1}.\\
The space $\X$ is then a   normed space w.r.to the   norm $\|\cdot\|_\vp$.

 \begin{defn} Let $\X$ be a complex vector space and $\Phi$ be a $\YY$-valued  inner product  on $\X$.
If $\X$ is complete w.r. to the  norm $\|\cdot\|_\vp$, then $\X$ is called a {\em  Banach space with $\YY$-valued   inner product} or for short a {\em $B_\YY$-space}. 
\end{defn}
If $\vp$ is not faithful, we can consider the subspace of $\X$ $$\mathfrak{N}_\Phi=\{x_1\in\X:\, \vp(x_1,x_2)=0_\YY, \forall x_2\in \X\}.$$
  By Proposition \ref{prop: 3.6 cases}, then
\begin{equation*}
    \mathfrak{N}_\Phi=\{x\in\X:\, \vp(x,x)=0_\YY\}.
\end{equation*}

We denote by $\Lambda_\vp(x)$ the coset of $\X/\mathfrak{N}_\Phi$ containing $x\in \X$; i.e., $\Lambda_\vp(x)=x+\mathfrak{N}_\Phi$ and define 
a  $\YY$-valued  inner product on $\X/\mathfrak{N}_\Phi$  as follows: 
 \begin{equation*}
\ip{\Lambda_\vp(x_1)}{ \Lambda_\vp(x_2)}_\vp:=\vp(x_1,x_2),\quad x_1,x_2\in\X.
\end{equation*}The associated   norm is:\begin{equation*}
\|\Lambda_\vp(x)\|_\vp:=\sqrt{\|\vp(x,x)\|_\YY},\quad x\in\X.\end{equation*}
 The quotient space $\X/\mathfrak{N}_\Phi=\Lambda_\vp(\X)$ is a   normed space (see \cite{BIvT1}).
 
 \medskip

Let $\mathcal{K}$ be a   $B_\YY$-space and $D(T)$ be a dense subspace of $\mathcal{K}$.
A linear map $T:D(T)\to \mathcal{K}$ is said {\em $\vp$-adjointable} if there exists a linear map $T^*$ defined on a subspace $D(T^*)\subset \mathcal{K}$ such that
$$\vp(T\xi,\eta)= \vp(\xi, T^*\eta), \quad \forall \xi \in D(T), \eta\in D(T^*).$$

Let $\D$ be a dense subspace of $\mathcal{K}$ and let us consider the following families of linear operators acting on $\D$:
\begin{align*}		{\LDK}&=\{T \mbox{ $\vp$-adjointable}, D(T)=\D;\; D(T^*)\supset \D\} \\	{\Lc^\dagger(\D)}&=\{T\in \LDK: T\D\subset \D; \; T^*\D\subset \D\}.\end{align*}
The involution in $\LDK$ is defined by		$T^\dag := T^*\upharpoonright \D$, the restriction of $T^*$, the $\vp$-adjoint of $T$, to $\D$.
The set $\Lc^\dagger(\D)$ is a  *-algebra.

\begin{rem} \label{rem_closable} If $T\in \LDK$ then $T$ is closable and its $\Phi$-adjoint $T^*$ is closed (it can be shown similarly as in \cite[Remark 2.8]{BIvT1} due to Proposition \ref{prop: 3.6 cases}). Moreover, the space
$\LDK$ is a {\em partial *-algebra} \cite{ait_book}  with respect to the following operations: the usual sum $T_1 + T_2 $,
the scalar multiplication $\lambda T$, the involution $ T \mapsto T\ad := T^* \up {\D}$ and the \emph{(weak)}
partial multiplication $\mult$ of two operators $T_1,T_2\in \LDK$
defined whenever there  exists $W\in \LDK$ such that
$$\Phi(T_2 \xi,T_1^\dag\eta)= \Phi(W\xi,\eta), \quad \forall \xi,\eta \in \D.$$
Due to the density of $\D$ in $\mathcal{K}$, the element $W$, if it exists, is unique. We put $W=T_1\mult T_2$.
\end{rem}

 \begin{defn} \label{defn_starrepmod} 
Let $(\A,\A_0)$ be a quasi *-algebra with unit $\id$.   Let $\mathcal{D}$ be a dense subspace
of a certain $B_\YY$-space $\mathcal{K}$ with
$\YY$-valued    inner product $\ip{\cdot}{\cdot}_\mathcal{K}$.   A linear map $\pi$ from $\A$ into  ${\mathcal L}\ad(\mathcal{D},\mathcal{K})$ is called  a \emph{*-representation} of  $(\A, \A_0)$,
if the following properties are fulfilled:
\begin{itemize}
\item[(i)]  $\pi(a^*)=\pi(a)^\dagger:=\pi(a)^*\upharpoonright\mathcal{D}, \quad \forall \ a\in \A$;
\item[(ii)] for $a\in \A$ and $c\in \A_0$, $\pi(a)\mult\pi(c)$ is well-defined and \linebreak {$\pi(a)\mult \pi(c)=\pi(ac)$}.
\end{itemize}
We assume that for every *-representation $\pi$ of $(\A,\A_0)$,  $\pi(\id)={\idop_{\mathcal{D}}}$, the
identity operator on  the space $\mathcal{D}$.\\
 The *-representation $\pi$  is said to be\begin{itemize}
\item {\em closable} if there exists $\widetilde{\pi}$ the closure of $\pi$, defined as $\widetilde{\pi}(a)=\overline{\pi(a)}\upharpoonright{\widetilde{\D}}$ where $\widetilde{\D}$ is the completion under the graph topology $t_\pi$ defined by the seminorms $\xi\in\mathcal{D}\to \|\xi\|_\mathcal{K}+\|\pi(a)\xi\|_\mathcal{K}$, $a\in\A$, with $ \|\cdot\|_\mathcal{K}$ the    norm induced by the   inner product on $\mathcal{K}$; 
\item \emph{closed} if  $\mathcal{D}[t_\pi]$ is complete; 
\item  \emph{cyclic} if there  exists $\xi\in\mathcal{D}$ such that $\pi(\A_0)\xi$ is dense in $\mathcal{K}$ in its  norm topology. In this case $\xi$ is called {\em cyclic vector}.
\end{itemize}

\end{defn}

\begin{defn}\label{defn: IA} Let $(\A,\A_0)$ be a quasi *-algebra. We denote by $\IA$ the set of all $\YY$-valued  positive sesquilinear  maps on $\A \times \A$ with the following properties:
	\begin{itemize}
	\item[(i)] $\Lambda_\vp(\A_0)=\A_0/\mathfrak{N}_\Phi$ is dense in  the completion $\widetilde{\A}$ of $\A$  w.r. to the  norm $\|\cdot\|_\vp$;
	\item[(ii)]
	 $\vp(ac,d)=\vp(c, a^*d), \quad \forall \ a \in \A, \ c,d \in \A_0$ (left-invariant).  \end{itemize} \end{defn}

If $\A=\A_0$ and $\omega:\A\to\YY$ is a positive linear map, then $\Phi:\A\times\A\to\YY$ given by $$\Phi(a,b)=\omega(b^*a), \quad\forall a,b\in\A$$ satisfies the conditions of Definition \ref{defn: IA}, i.e. $\Phi\in\IA$.

\begin{prop}\label{prop_rep}
\vspace{-1mm} Let $(\A,\A_0)$  be a  quasi *-algebra with unit $\id$ and $\vp$ be a $\YY$-valued left-invariant positive  sesquilinear  map on $\A \times \A$.
The following statements are equivalent:
\begin{itemize}
\item[{\em (i)}]$\vp\in\IA$;
\item[{\em (ii)}] there exist a $B_\YY$-space $\mathcal{K}_\Phi$ with $\YY$-valued  inner product  $\ip{\cdot}{\cdot}_{\mathcal{K}_\Phi}$,
a dense subspace $\D_\Phi\subseteq\mathcal{K}_\Phi$ and a closed cyclic *-representation $\pi:\A\to{\mathcal L}^\dag(\D_\Phi,\mathcal{K}_\Phi)$ with cyclic vector $\xi_\vp$ such that $$\ip{\pi(a)\xi}{\eta}_{\mathcal{K}_\Phi}=\ip{\xi}{\pi(a^*)\eta}_{\mathcal{K}_\Phi}, \quad \forall \xi,\eta\in\D_\Phi, a\in\A$$ and such that  $$\vp(a,b)=\ip{\pi(a) \xi_\vp}{\pi(b) \xi_\vp}_{\mathcal{K}_\Phi}, \quad \forall  a,b\in\A.$$ 
\end{itemize}
\end{prop}
\begin{proof}The proof proceeds along the lines of that one of \cite[Theorem 3.2]{BIvT1}, due to the Cauchy-Schwarz inequality for positive $\YY$-valued sesquilinear maps (Proposition \ref{prop: 3.6 cases}) and due to Remark \ref{rem_closable}. \end{proof}

\begin{rem}\label{re: bounded pos sesq}
     By the same arguments as in \cite[Corollary 3.5]{BIvT1} 
  one can show that every $\YY$-valued, {\em bounded},  left-invariant, positive sesquilinear map on a unital normed quasi *-algebra belongs to $\IA$.\end{rem}

 In the following example we construct some $\YY$-valued, bounded, left-invariant, positive sesquilinear maps.
    \begin{ex}\label{ex: 2.3} Let $\rho$ be a finite trace on a von     Neumann algebra $\MM$. Let $W\in L^\infty(\rho)$, with $W\geq 0$. Let $k\in C([0,\|W\|]\times[0,\|W\|])$ be such that $k\geq0$. Then, for each $x\in[0,\|W\|]$, the function $\eta_x:[0,\|W\|]\to\mathbb{C}$ defined by $\eta_x(t)=k(x,t)$ is a continuous, positive function on $[0,\|W\|]$. Therefore, by the functional calculus, $\eta_x(W)$ defines a positive operator in $L^\infty(\rho)$. Let us define, for every $x\in[0,\|W\|]$ and $X,Y\in L^2(\rho)$,
        \begin{equation*}
        \varphi(X,Y)(x)=\rho(X\eta_x(W)Y^*).\end{equation*} 
        Then,  $\varphi(X,Y)\in C([0,\|W\|])$ for all $X,Y\in L^2(\rho)$. Moreover, $\varphi:L^2(\rho)\times L^2(\rho)\to C([0,\|W\|]) $ is a bounded, left-invariant, positive sesquilinear map (see \cite{BIvT1}).\\
 Let $f:[0,1]\to [0,\|W\|]$ be  measurable. For every $X,Y \in L^2(\rho)$, let $\psi:L^2(\rho)\times L^2(\rho)\to L^2([0,1])$ defined by $\psi(X,Y)=\phi(X,Y)\circ f$. 
Then one can check that $\psi$ is a bounded, left-invariant, positive sesquilinear map.\\ 
%
    Finally, if $(\A,\A_0)$ is a unital quasi *-algebra and $\Psi:\A\times\A\to L^1(\Omega)$ is a bounded left-invariant  positive sesquilinear map, then the induced map $\widetilde{\Psi}:\A\times\A\to M(\Omega)$ given by $d\widetilde{\Psi}(a,b)=\Psi(a,b)d\mu$, for all $a,b\in\A$ and some fixed positive measure $\mu\in M(\Omega)$, is also a bounded left-invariant  positive sesquilinear map.
\end{ex}

\begin{cor}\label{moreGNS alg}  Let $\A$ be a *--algebra with unit
$\id$ and let $\omega$ be a positive  linear $\YY$-valued map on $\A$. Then, there
exists a  $B_\YY$-space $\mathcal{K}_\Phi$ whose  norm is induced by a $\YY$-valued   inner product $\ip{\cdot}{\cdot}_{\mathcal{K}_\Phi}$, a dense subspace $\D_\omega\subseteq\mathcal{K}_\Phi$ and a closed cyclic *--representation $\ppi_\omega$ of $\A$ with domain $\D_\omega$, such that $$\omega(b^*ac)=\ip{\ppi_\omega(a) \Lambda_\omega(c)}{\Lambda_\omega(b)}_{\mathcal{K}_\Phi},\quad\forall a,b,c\in\A.$$ Moreover, there exists  a  cyclic vector $\eta_\omega$,
such that $$ \omega(a)=\ip{{\ppi}_\omega(a)\eta_\omega}{\eta_\omega}_{\mathcal{K}_\Phi}, \quad \forall \ a \in \A.$$ 
The representation is unique up to unitary equivalence. 
\end{cor}\begin{proof}
    The statement can be proved analogously to \cite[Corollary 3.10]{BIvT1}, due to Proposition \ref{prop_rep}.
\end{proof}

 \section{Representations induced by completely positive sesquilinear maps}
\label{Sect_4}
In this section, we extend the results of Section \ref{sect_3} to the case of $\YY$-valued completely positive  sesquilinear maps.  
\\

 
Let  $\mathfrak{X}$ and $\YY$ be   vector spaces. We will denote by $\mathcal{S}_{\YY}(\mathfrak{X})$  the space of all $\YY$-valued sesquilinear maps on $\mathfrak{X}$.

\begin{defn}\label{defn: about Phi in SYX}
Let  $\mathfrak{X}$ be a normed vector space, $\YY$ an ordered Banach module over the  *-algebra $\YY_0$, with positive cone $\KK$ and 
 $(\A,\A_0)$ a normed quasi *-algebra with unit.    The sesquilinear map $\Phi:\A\times \A\to \mathcal{S}_{\YY}(\mathfrak{X})$  is called  \begin{itemize}
       \item {\em bounded} if there exists a constant $M>0$ such that $$\|\Phi(a,b)(x_1,x_2)\|_{\YY}\leq M \|a\|\|b\|\|x_1\|\|x_2\|, \quad \forall a,b\in\A, x_1,x_2\in\mathfrak{X};$$
       \item {\em left-invariant} if $\Phi(ac,d)=\Phi(c,a^*d), \quad\forall a\in\A, c,d\in\A_0$;
       \item  {\em completely positive} if for every $N \in {\mb N}$, $a_1,\dots, a_N \in \A$, \\ \mbox{$x_1,\dots, x_N\in \X$},
$$\sum_{i,j=1}^N \Phi(a_i, a_j)(x_i,x_j) \in \KK.$$
   \end{itemize} 
\end{defn}

The   Stinespring Theorem is the main result on completely  positive linear maps.  In the following we consider  certain  {\em completely positive sesquilinear maps} taking values on the space of all $\YY$-valued sesquilinear maps on a vector space, see Theorem \ref{thm: Steinspring}  below.\\

 Let  $(\A,\A_0)$ be a normed quasi *-algebra with unit $\id$,  $\mathfrak{X}$ be a normed complex vector space. Let the cone $\KK$ in $\YY$ be closed and  
 $\Phi:\A\times \A\to \mathcal{S}_{\YY}(\X)$ be a left-invariant positive sesquilinear map. Consider  the algebraic tensor product $\A\otimes \X$ and its subset $$\mathcal{N}_\Phi=\left\{\sum_{i=1}^n a_i\otimes x_i\in\A\otimes \mathfrak{X}|\, \sum_{i=1}^n\sum_{j=1}^n\Phi(a_i, a_j)(x_i,x_j)=0_\YY\right\}.$$
 

Hence, it is easy to check that the quotient space $(\A\otimes \X)/\mathcal{N}_\Phi$ is a   normed space.
\begin{thm}\label{thm: Steinspring} 
Let  $\mathfrak{X}$ be a normed complex vector space. 
Let $(\A,\A_0)$ be a normed quasi *-algebra with unit $\id$  and  $\Phi:\A\times \A\to \mathcal{S}_{\YY}(\X)$  a left-invariant  sesquilinear map. Let $\Phi$ be bounded with bound $M>0$. Then, $\Phi$ is completely positive if and only if there exist a  $B_\YY$-space $\mathcal{K}_\Phi$,  a dense subspace $\D_\Phi$ of $\mathcal{K}_\Phi$, a closed *-representation $\pi$ of $\A$ in ${\mathcal L}\ad(\D_\Phi,\mathcal{K}_\Phi)$ and a bounded linear operator $V:\mathfrak{X}\to\D_\Phi$  such that $\pi(\A)V\X=(\A\otimes \X)/\mathcal{N}_\Phi$ and, for all $a,b\in \A$ and $x_1,x_2\in \mathfrak{X}$, it holds that $$\Phi(a,b)(x_1,x_2)=\ip{\pi(a)Vx_1}{\pi(b)Vx_2}_{\mathcal{K}_\Phi}.$$ In this case $\|V\|^2\leq M\|\id\|^2$. Moreover, the  triple $(\pi,V,\D_\Phi)$  is     such that $\overline{\pi(\A)V\X}=\mathcal{K}_\Phi$. 
\end{thm}
 \begin{proof} The proof  goes along the lines of that of \cite{BIT} where more traditional sesquilinear forms were considered thanks to Proposition \ref{prop: 3.6 cases} and Remark \ref{rem_closable}. We omit the details. 
\end{proof}

 By Theorem \ref{thm: Steinspring},  and similar arguments to those in the proof of Proposition \ref{prop: 3.16 sequences}, we deduce the following
 
 \begin{cor}\label{cor: infinite sums}
 Let  $\mathfrak{X}$ be a normed complex vector space. 
Let $(\A,\A_0)$ be a normed quasi *-algebra with unit $\id$  and        $\{\Phi_n\}_n$ be a sequence of bounded, left-invariant, completely positive sesquilinear maps $\Phi_n:\A\times\A\to \mathcal{S}_\YY(\X)$. Let $\{a_n\}_n$, $\{\widetilde{a}_n\}_n$ be sequences in $\A$ and $\{x_n\}_n$, $\{\widetilde{x}_n\}_n$ be sequences in $\X$ such that the series $\sum_{n=1}^\infty \Phi_n(a_n,a_n)(x_n,x_n)$ and $\sum_{n=1}^\infty \Phi_n(\widetilde{a}_n,\widetilde{a}_n)(\widetilde{x}_n,\widetilde{x}_n)$ are convergent in $\YY$. Then the series
       $\sum_{n=1}^\infty \Phi_n(a_n,\widetilde{a}_n)(x_n,\widetilde{x}_n)$ is convergent and
     \begin{eqnarray}\label{eq: Phi_n} 
	&& \hspace{3mm} \left\|\sum_{n=1}^\infty \Phi_n(a_n,\widetilde{a}_n)(x_n,\widetilde{x}_n)\right\|_\YY \\ &&\nonumber \quad\quad\leq \left\|\sum_{n=1}^\infty \Phi_n(a_n,a_n)(x_n,x_n)\right\|_\YY^{1/2}\left\|\sum_{n=1}^\infty \Phi_n(\widetilde{a}_n,\widetilde{a}_n)(\widetilde{x}_n,\widetilde{x}_n)\right\|_\YY^{1/2}.\end{eqnarray}
 \end{cor}

In the following we will refer to a  triple $(\pi, V, \D_\Phi)$ as in  Theorem \ref{thm: Steinspring} as a {\em Stinespring  triple decomposing $\Phi$}.

\begin{defn}
    Let $\X_1,\X_2$ be two   $B_\YY$-spaces with respect to the   norms  $\|\cdot\|_\Phi$ and $\|\cdot\|_\Psi$, respectively, induced by two $\YY$-valued   inner products $\Phi:\X_1\times\X_1\to \YY$ and $\Psi:\X_2\times\X_2\to \YY$. A surjective operator $U:\X_1\to\X_2$ is said {\em unitary} if $$\ip{U\xi}{U\eta}_\Psi=\ip{\xi}{\eta}_\Phi, \quad \forall \xi,\eta\in\X_1.$$
\end{defn}
    In analogy to Corollary 3.10 in \cite{BIT}, we can prove the following

\begin{prop}\label{prop: unitary equivalence}
Let  $\mathfrak{X}$ be a normed complex vector space. 
Let $(\A,\A_0)$ be a normed quasi *-algebra with unit $\id$ and  $\Phi:\A\times \A\to \mathcal{S}_{\YY}(\X)$ be a bounded left-invariant   sesquilinear map.  If $\Phi$ is completely positive then  $\pi$ and $V$ in the  triple $(\pi, V, \D_\Phi)$ are uniquely determined by $\Phi$ up to unitary equivalence, i.e., if  $(\pi, V, \D_\Phi)$ and $(\pi_1, V_1, \mathcal{E}_\Phi)$  are two Stinespring  triples decomposing $\Phi$, there exists a unitary operator $U$ such that  $UV = V_1$, 
$U\D_\Phi = \mathcal{E}_\Phi$ and $\pi(a) = U^{-1}\pi_1(a)U$, for all $a\in\A$. 
\end{prop}

\begin{rem} By Proposition \ref{prop: unitary equivalence} all the Stinespring  triples $(\pi,V,\D_\Phi)$ decomposing $\Phi$ are
    such that $\overline{\pi(\A)V\X}=\mathcal{K}_\Phi$. 
\end{rem}

\medskip

Following \cite[Section 3.5]{stormer} we now provide a generalization  of the Radon-Nikodym theorem for completely positive sesquilinear maps with values in the space $\mathcal{S}_{\YY}(\X)$ of $\YY$-valued sesquilinear maps on $\X\times \X$.

\begin{prop}\label{prop: lemma 3.5.2}      Let  $\mathfrak{X}$ be a normed complex vector space, 
$(\A,\A_0)$ be a normed quasi *-algebra with unit $\id$ and  $\Psi,\Phi:\A\times \A\to \mathcal{S}_{\YY}(\X)$ be  bounded,  left-invariant,   completely  positive sesquilinear maps.    
     Suppose that there exists $\gamma>0$ such that, for every $N\in\mathbb{N}$, for all $a_1,\dots,a_N\in\A, x_1,\dots,x_N\in\X$  \begin{equation*}
           \sum_{i,j=1}^N\Psi(a_i,a_j)(x_i,x_j)\leq \gamma \sum_{i,j=1}^N\Phi(a_i, a_j)(x_i,x_j).
     \end{equation*}
     Let $(\pi_\Psi,V_\Psi, \D_\Psi)$ and $(\pi_\Phi,V_\Phi, \D_\Phi)$ be two Stinespring triples decomposing  $\Psi$ and $\Phi$, respectively.
Then, there exists a linear operator
$T : \pi_\Phi(\A)V_\Phi\X\to\mathcal{K}_\Psi$   such that:
\begin{itemize}
    \item[(i) ] $TV_\Phi = V_\Psi$;
    \item[(ii) ] $T\pi_\Phi(a) = \pi_\Psi(a)\mult T$ on $\pi_\Phi(\A_0)V_\Phi\X$, for every $a\in\A$, \end{itemize} and for all $a,b\in\A$, $x_1,x_2\in\X$, $$\Psi(a,b)(x_1,x_2)=\ip{T\pi_\Phi(a)V_\Phi x_1}{T\pi_\Phi(b)V_\Phi x_2}_{\mathcal{K}_\Psi}.$$ Moreover, if the norm $\|\cdot\|_\YY$ in $\YY$ is order-preserving, then $T$ is bounded, defined on $\mathcal{K}_\Phi=\overline{\pi_\Phi(\A)V_\Phi\X}$ and $\|T\|\leq\sqrt{ \gamma}$. In this case, for every $a\in\A$, $T\pi_\Phi(a)=\pi_\Psi(a)\mult T$ on the whole $\mathcal{K}_\Phi$.
\end{prop} 
\begin{proof} Let $N\in\mathbb{N}$ and  $a_1, \dots,a_N\in\A$  and $x_1,\dots, x_N\in\X$, then 
\begin{align}\label{eq: ineq}
&\left<\sum_{i=1}^N\pi_\Psi(a_i)V_\Psi x_i\Bigg|\sum_{j=1}^N\pi_\Psi(a_j)V_\Psi x_j\right>_{\mathcal{K}_\Psi}= \sum_{i,j=1}^N\Psi(a_i,a_j)(x_i,x_j)\\
&\hspace{1cm}\leq\gamma\sum_{i,j=1}^N\Phi(a_i,a_j)(x_i,x_j)  \nonumber\\& \hspace{1cm}\nonumber=\gamma\left<\sum_{i=1}^N\pi_\Phi(a_i)V_\Phi x_i\Bigg|\sum_{j=1}^N\pi_\Phi(a_j)V_\Phi x_j\right>_{\mathcal{K}_\Phi}.
\end{align}
If $\sum_{i=1}^N\pi_\Phi(a_i)V_\Phi x_i=0$, then also $\sum_{i=1}^N\pi_\Psi(a_i)V_\Psi x_i=0$, hence,  we can define a  linear operator  $T$  by \begin{equation}
     \label{eq: equality}T\sum_{i=1}^N\pi_\Phi(a_i)V_\Phi x_i=\sum_{i=1}^N\pi_\Psi(a_i)V_\Psi x_i,
 \end{equation}  for all $a_i\in\A,\,x_i\in\X, i\in\{1,\dots,N\}$.
 Moreover, by taking $a=\id\in\A_0$ then $TV_\Phi=V_\Psi$, since $\pi_\Phi(\id)=\idop_{\D_\Phi}$ and $\pi_\Psi(\id)=\idop_{\D_\Psi}$.\\  Then $T\mult\pi_\Phi(a) = \pi_\Psi(a)\mult T$, for every $a\in\A$. 
 Indeed, if now $a\in\A$, $c\in\A_0$, $x\in \X$, by \eqref{eq: equality}
 \begin{equation*}
     \begin{split}
     T(\pi_\Phi(a)\mult\pi_\Phi(c))V_\Phi x&= T({\pi_\Phi(a)^\dag}^*\pi_\Phi(c))V_\Phi x=
     (T{\pi_\Phi(a)^\dag}^*)\pi_\Phi(c)V_\Phi x,    \end{split}
 \end{equation*}on the other hand,\begin{equation*}
     \begin{split}
T(\pi_\Phi(a)\mult\pi_\Phi(c))V_\Phi x&=T\pi_\Phi(ac)V_\Phi x=\pi_\Psi(ac)V_\Psi x\\&={\pi_\Psi(a)^\dag}^*\pi_\Psi(c)V_\Psi x={\pi_\Psi(a)^\dag}^*( T\pi_\Phi(c)V_\Phi x)\\&=({\pi_\Psi(a)^\dag}^* T)\pi_\Phi(c)V_\Phi x=(\pi_\Psi(a)\mult T)\pi_\Phi(c)V_\Phi x.
     \end{split}
 \end{equation*} 
Finally,  if $y_1\leq y_2$, with $y_1,y_2\in\KK$ implies that also $\|y_1\|_\YY\leq \|y_2\|_\YY$
then, by \eqref{eq: ineq}, we have that $T$ is bounded, hence it extends to the closure $\mathcal{K}_\Phi=\overline{\pi_\Phi(\A)V_\Phi\X}$ and  $\|T\|\leq\sqrt{ \gamma}$. Since  the set $\pi_\Phi(\A_0)V_\Phi\X=(\A_0\otimes\X)/\mathcal{N}_\Phi$ is dense in $\mathcal{K}_\Phi$, we conclude that $T\pi_\Phi(a)=\pi_\Psi(a)\mult T$ on $\mathcal{K}_\Phi$, for every $a\in\A$.  The last statement  follows by the very construction of $T$.\end{proof}

 A Radon Nikodym-like theorem holds for positive sesquilinear maps $\Phi,\Psi$  and not just for completely positive sesquilinear maps.
\begin{prop}     \label{prop: RN positive}
  Let $\mathfrak{X}$ be a normed vector space and let  
  the norm $\|\cdot\|_\YY$ preserve the order in $\YY$.
 Let   $(\A,\A_0)$ be a normed quasi *-algebra with unit $\id$ and  $\Phi,\Psi:\A\times \A\to \mathcal{S}_{\YY}(\X)$ be  bounded  left-invariant     positive sesquilinear maps  such that  
     \begin{equation*}
          \Psi(a,a)(x,x)\leq \gamma\, \Phi(a, a)(x,x), \quad \forall a\in\A, x\in\X
     \end{equation*} for some  $\gamma>0$. Let $\mathcal{K}_\Phi$ and $\mathcal{K}_\Psi$ be   $B_\YY$-spaces,  $\D_\Phi$ and $\D_\Psi$  dense subspaces of $\mathcal{K}_\Phi$ and $\mathcal{K}_\Psi$, respectively,  $\pi_\Phi:\A\to{\mathcal L}^\dag(\D_\Phi,\mathcal{K}_\Phi)$ and $\pi_\Psi:\A\to{\mathcal L}^\dag(\D_\Psi,\mathcal{K}_\Psi)$ closed cyclic *-representations of $\Phi$ and $\Psi$, respectively, with cyclic vectors $\xi_\vp$ and $\xi_\Psi$, respectively,  as in Proposition \ref{prop_rep}.
 Then there exists a unique operator  $T:\mathcal{K}_\Phi\to \mathcal{K}_\Psi$, 
with $\|T\| \leq  \sqrt{\gamma}$ such that for all $a,b\in \A$,  $$\Psi(a,b)=\ip{T\pi_\Phi(a)\xi_\Phi}{T\pi_\Phi(b)\xi_\Phi}_{\mathcal{K}_\Psi}.$$ 
Moreover,
$T\pi_\Phi(a)=\pi_\Psi(a)\mult T$ on $\pi_\Phi(\Ao)\xi_\vp$, for all $a\in\A$.
     \end{prop}
     \begin{proof}         
The proof goes along the lines of that  of Proposition \ref{prop: lemma 3.5.2}, by applying the representations of $\Phi$ and $\Psi$ of Proposition \ref{prop_rep} and by letting $T\pi_\Phi(a)\xi_\Phi=\pi_\Psi(a)\xi_\Psi$ for all $a\in\A$, see also Remark \ref{re: bounded pos sesq}. 
\end{proof}

We give now some examples. 

\begin{ex}
    Let $\YY$ be either a $C^*$-algebra or an ordered Banach module over a *-algebra, satisfying $(D1)-(D2)$  and let $\widetilde{\Phi}:\A\times \A\to\YY$ be a bounded left-invariant positive sesquilinear map. Define the map $\Phi:\A\times \A\to \mathcal{S}_{\YY}(\A_0)$ by $$\Phi(a,b)(c,d)=\widetilde{\Phi}(ac,bd), \quad\forall a,b\in\A, c,d\in \A_0.$$ Then $\Phi$ is  bounded, left-invariant, completely positive and sesquilinear.
\end{ex}

\begin{ex}\label{ex: esem 4.10}  Let $\MM$ be a von Neumann algebra with a finite trace $\rho$ and let $W$ and $\eta_x(W)$ be as in Example \ref{ex: 2.3}. Consider a separable Hilbert space $\H$ and the space $L^2([0,\|W\|], \B(\H))$ w.r. to the Gel'fand-Pettis integral (see \cite{Jocic} and \cite[Example 2.26]{BIvT1}). Fix  $F\in L^2([0,\|W\|], \B(\H))$ with $F(t)\geq0$ for a.e. $t\in [0,\|W\|]$ and  $T\in\B_2(\H)$ and consider $\Phi:L^2(\rho)\times L^2(\rho)\to \mathcal{S}_{\B_1(\H)}(L^\infty(\rho))$ given by $$\Phi(A,B)(X_1,X_2)=T^*\left(\int_0^{\|W\|}\rho(X_2^*B^*AX_1\eta_x(W))) F(x)dx\right)T,$$ for every $A,B\in L^2(\rho)$, $X_1,X_2\in L^2(\rho).$ Then $\Phi$ is a bounded, left-invariant, completely positive sesquilinear map.
\end{ex}


\subsection{Applications to operator-valued maps}

In the following we will denote by $\B(\X,\YY)$ the space of bounded linear operators from $\X$ into $\YY$, where $\X$ is a general Banach space.

\begin{defn}\label{defn: compl pos w.r.t Gamma} Let $\X$  be a  Banach space, $\Gamma:\X\times\X\to\X$  a  bounded 
sesquilinear map. Let $(\A,\A_0)$ be a quasi *-algebra with unit $\id$.
   A sesquilinear map $\Phi:\A\times \A\to \B(\X,\YY)$ will be called 
       {\em completely positive w.r. to $\Gamma$} if for each $N\in\mathbb{N}$ and for every $a_1,\dots,a_N\in\A$, $x_1,\dots, x_N\in\X$ it is $$\sum_{i,j=1}^N\Phi(a_i,a_j)(\Gamma(x_i,x_j) )\in\KK.$$
\end{defn}
\begin{prop}\label{prop: von Neumann}
Let   $\X$  be a  Banach space, $\Gamma:\X\times\X\to\X$  a bounded 
sesquilinear map. Let $(\A,\A_0)$ be a normed quasi *-algebra with unit $\id$ and $\Phi:\A\times\A\to \B(\X,\YY)$ be  a bounded, left-invariant,  completely positive sesquilinear map.    Then,  there exist a dense subspace $\D_\Phi$ of a     $B_\YY$-space $\mathcal{K}_\Phi$, a closed  *-representation $\pi$ of $\A$ in ${\mathcal L}\ad(\D_\Phi,\mathcal{K}_\Phi)$  and a bounded linear operator $V:\X\to \D_\Phi$ with $\|V\|^2\leq \|\Phi(\id,\id)\|\|\Gamma\|$ such that we have $$
    \Phi(a,b)(\Gamma(x_1,x_2) )=\ip{\pi(a)Vx_1}{\pi(b)Vx_2}_{\mathcal{K}_\Phi}, \quad\forall a,b\in\A,\,\, x_1,x_2\in\X.
$$ Moreover, for all $a,b\in\A$ and $x_1,x_2\in\X$, we have  \begin{multline}\label{eq: inequ 1 prop 4.12}\|\Phi(a,b)(\Gamma(x_1,x_2))\|_\YY\\\leq  \|\Phi(a,a)(\Gamma(x_1,x_1))\|_\YY^{\frac{1}{2}}\|\Phi(b,b)(\Gamma(x_2,x_2))\|_\YY^{\frac{1}{2}}. \end{multline}
If $B_1$ denotes the unit ball in $\X$ and $\Gamma(B_1 \times B_1) = B_1$, then for all $a,b\in\A$  we have \begin{equation}\label{eq: ineq 2 prop 4.12}\|\Phi(a,b)\|\leq\|\Phi(a,a)\|^{1/2}\|\Phi(b,b)\|^{1/2}.\end{equation} \end{prop}
   \begin{proof}
       Let $\mathcal{S}_{\YY}(\X)$ be the space of all $\YY$-valued sesquilinear maps on $\X$ and $\widetilde{\Phi}:\A\times\A\to \mathcal{S}_{\YY}(\X)$ be given by: $$\widetilde{\Phi}(a,b)(x_1,x_2):=\Phi(a,b)(\Gamma(x_1, x_2)),\quad \forall a,b\in \A, \,\, x_1,x_2\in\X.$$ Then, $\widetilde{\Phi}$ is a left-invariant, completely positive sesquilinear map on $\YY$ since it inherits these properties from $\Phi$. Moreover $\widetilde{\Phi}$ is bounded. By Theorem \ref{thm: Steinspring}, the statement follows.  
       Since $\pi$ is a *-representation, $\pi(\id)=\idop_{\mathcal{K}_\Phi}$, the
identity operator on  the space $\mathcal{K}_\Phi$, hence,    
       for every $x\in\X$, we have that  \begin{equation*}
    \begin{split}
        \|\ip{Vx}{Vx}_{\mathcal{K}_\Phi}\|_\YY&= \|\ip{\pi(\id)Vx}{\pi(\id)Vx}_{\mathcal{K}_\Phi} \|_\YY=\|\ip{\id\otimes x}{\id\otimes x}_{\mathcal{K}_\Phi}\|_\YY\\&=\|(\Phi(\id,\id))(\Gamma(x,x))\|_\YY\leq \|\Phi(\id,\id)\|\|\Gamma(x,x)\|\\&\leq \|\Phi(\id,\id)\|\|\Gamma\|\|x\|^2.
    \end{split}
\end{equation*}
        The inequality \eqref{eq: inequ 1 prop 4.12} follows from  Proposition  
        \ref{prop: 3.6 cases} applied to the obtained representation, whereas  \eqref{eq: ineq 2 prop 4.12} follows by taking supremums over $B_1$ in \eqref{eq: inequ 1 prop 4.12} and using that $\Gamma (B_1\times B_1)=B_1$.
   \end{proof}

       The following example shows  maps  $\Phi$ and $\Gamma$ satisfying the hypotheses of Proposition \ref{prop: von Neumann}.
\begin{ex}\label{ex: compl positive} Take $\X=\MM$ a von Neumann algebra with a finite trace $\rho$ and $\Gamma(T_1,T_2)=T_1T_2^*$, for all $T_1,T_2\in\MM$. Since $\MM$ is unital, it follows that $\Gamma(B_1\times B_1)=B_1$. Let $W$ and $\eta_x(W)$ be as in Example \ref{ex: 2.3}. For every $T\in\MM$ and $x\in[0,\|W\|]$, let $$(\varphi(A,B)(T))(x)=\rho(A\,\eta_x(W)\,T\,\eta_x(W)\,B^*),\quad A,B\in L^2(\rho).$$ By some calculations it is not hard to see that $\varphi(A,B)(T)(\cdot)$ is a continuous function on $[0,\|W\|]$ and that $\varphi$ is a bounded, left-invariant, sesquilinear map from $L^2(\rho)\times L^2(\rho)$ into {  $\B(\MM,C([0,\|W\|]))$, which is completely positive w.r. to $\Gamma$.}
%
 Moreover, if $\widetilde{\MM}$ is another von Neumann algebra with a semi-finite trace $\widetilde{\rho}$ and if we choose some $\widetilde{W}\in\widetilde{\MM}$, with $\widetilde{W}\geq0$ and $\|\widetilde{W}\|=\|W\|$, then by the functional calculus, $\varphi(A,B)(T)(\widetilde{W})\in\widetilde{\MM}$, for all $A,B\in L^2(\rho)$ and $T\in\MM$. Hence, the map $\Phi:L^2(\rho)\times L^2(\rho)\to \B(\MM, \widetilde{\MM}^*)$  given by $$(\Phi(A,B)(T))(S)=\tilde{\rho}(S(\varphi(A,B)(T))(\widetilde{W}))$$ for all $T\in\MM$, $A,B\in L^2(\rho)$ and $S\in\widetilde{\MM}$ is  another example of a bounded, left-invariant sesquilinear map which is completely positive w.r. to $\Gamma$. Further, if  $G\in L^1(\widetilde{\rho})$, then the map $\widetilde{\Phi}: L^2(\rho)\times L^2(\rho)\to \B(\MM, L^1(\widetilde{\rho}))$ given by $$\widetilde{\Phi}(A,B)(T)=G(\varphi(A,B)(T)(\widetilde{W}))G^* $$ is also an example of bounded left-invariant sesquilinear map which is completely positive w.r. to $\Gamma$.
\end{ex}
Motivated by the notion of the classical complete positivity, we introduce now the following generalization.
\begin{defn}\label{defn: compl pos wrt psi}
    Let $\X,\Z$ be Banach spaces, $\Psi:\Z\times\X\to\YY$ and $\Phi:\A\times\A\to \B(\X,\Z)$ be sesquilinear maps.  The map $\Phi$ is said {\em completely positive w.r. to $\Psi$} if for every $N\in\mathbb{N}$ and all $a_1,\dots,a_N\in\A$, $x_1,x_2,\dots,x_N\in\X$ it holds that $$\sum_{i,j=1}^N\Psi(\Phi(a_i,a_j)x_i,x_j)\in\KK.$$
\end{defn} 

\begin{rem}
    If $\X=\Z=\H$ is a Hilbert space, $\YY=\mathbb{C}$, $\A$ is a unital $C^*$-algebra, $\Psi$ is inner product on $\H$, $\phi:\A\to\B(\H)$ is a linear map and $\Phi(a,b)=\phi(b^*a)$, for all $a,b\in\A$, then Definition \ref{defn: compl pos wrt psi} reduces to the definition of the classical complete positivity of $\phi$ and $\Phi$.
\end{rem}

\begin{prop}
    Let $\X,\Z$ be Banach spaces, $\Psi:\Z\times\X\to\YY$ and $\Phi:\A\times\A\to \B(\X,\Z)$ be sesquilinear maps. If $\Phi$ and $\Psi$ are bounded and $\Phi$ is left-invariant and completely positive w.r. to $\Psi$, then there exists a  dense subspace $\D_\Phi$ of a     $B_\YY$-space $\mathcal{K}_\Phi$, a closed  *-representation $\pi$ of $\A$ in ${\mathcal L}\ad(\D_\Phi,\mathcal{K}_\Phi)$  and a bounded linear operator $V:\X\to \D_\Phi$ with $\|V\|^2\leq \|\Phi(\id,\id)\|\|\Psi\|$ such that  we have $$\Psi(\Phi(a,b)x_1,x_2)=\ip{\pi(a)Vx_1}{\pi(b)Vx_2}_{\mathcal{K}_\Phi}, \quad\forall a,b\in\A,\,\, x_1,x_2\in\X.$$ Moreover, $$\|\Psi(\Phi(a,b)x_1,x_2)\|_\YY\leq\|\Psi(\Phi(a,a)x_1,x_2)\|_\YY^{1/2}\|\Psi(\Phi(b,b)x_1,x_2)\|_\YY^{1/2}, $$ for all $a,b\in\A,\,\, x_1,x_2\in\X$.\\
    If $\|\Psi\|=1$ and for all $z\in\Z$ we have that $\|z\|_{\Z}=\sup_{\substack{x \in \X \\ \|x \|_{\X}\leq 1}}   \|\Psi(z,x)\|_\YY$, then for all $a,b\in\A$,  we have $$\|\Phi(a,b)\|\leq\|\Phi(a,a)\|^{1/2}\|\Phi(b,b)\|^{1/2}.$$ 
\end{prop}\begin{proof}
    As in the proof of Proposition \ref{prop: von Neumann}, we can consider the induced sesquilinear map $\widetilde{\Phi}:\A\times\A\to \mathcal{S}_\YY(\X)$ given by $$\widetilde{\Phi}(a,b)(x_1,x_2)=\Psi(\Phi(a,b)x_1,x_2), \quad \forall x_1,x_2\in\X$$ and deduce the first two statements by applying Theorem \ref{thm: Steinspring} and the obtained *-representation. The last inequality in the statements follows from the first one, by taking the supremums over the unit ball in $\Z$.
\end{proof}

\begin{ex}
     Let $\MM$ be a von Neumann algebra with a finite trace $\rho$ and let $W\in L^\infty(\rho)$. Put $\Z=\YY=L^1(\rho)$, $\X=L^\infty(\rho)$. If $\Psi:\Z\times\X\to\YY$ is given by $\Psi(Z,X)=X^*Z$, then $\Psi$ is bounded, $\|\Psi\|=1$ and $\|Z\|_{\Z}=\sup_{\substack{X \in \X \\ \|X \|_{\X}\leq 1}}   \|\Psi(Z,X)\|_\YY$, due to the fact that $L^\infty(\rho)\subset L^1(\rho)$ and that $L^\infty(\rho)=\MM$ is unital. Let $\Phi:L^2(\rho)\times L^2(\rho)\to \B(\X,\Z)$ be given by $$\Phi(A,B)(X)=W^*B^*AW X, \quad \forall A,B\in L^2(\rho), \,X\in\X=L^\infty(\rho).$$  Then $\Phi$ is a bounded, left-invariant  sesquilinear map which is completely positive w.r. to $\Psi$.
\end{ex}

\smallskip

\noindent{\bf{Acknowledgements:}} GB and CT acknowledge that this work has been carried out within the activities of Gruppo UMI Teoria dell’Appros-simazione e Applicazioni and of GNAMPA of the INdAM. SI is supported by the Ministry of Science, Technological Development and Innovations, Republic of Serbia, grant no. 451-03-66/2024-03/200029. 

\bibliographystyle{amsplain}

\begin{thebibliography}{99}
	\bibitem{ait_book} J.-P. Antoine, A. Inoue, C. Trapani, {\em Partial *-algebras
		and their Operator Realizations}, Kluwer, Dordrecht, 2002.
\bibitem{Arveson} W.B. Arveson, {\em Subalgebras of $C^*$-algebras}, Acta Math. {\bf 123}  (1969) 141-224.
        
	\bibitem{Asad} M.B. Asad,  {\em Stinespring’s theorem for Hilbert $C^*$-modules}, J. Oper. Theory {\bf 62} (2009) 235-238. 
	\bibitem{BIT} F. Bagarello, A. Inoue, C. Trapani, {\em Completely Positive Invariant Conjugate-Bilinear Maps in Partial *-Algebras},  Z. Anal. Anwend.
	{\bf  26} (2007), 313-330.
	\bibitem{BIvT1} G. Bellomonte, S. Ivkovi\'{c}, C. Trapani, {\it  GNS Construction for $C^*$-Valued Positive Sesquilinear Maps on a quasi *-algebra},   Mediterr. J. Math. {\bf 21}   (2024) art. n.  166 (22 pp.).
	\bibitem{BDIv} G. Bellomonte, B. Djordjevi\'c, S. Ivkovi\'{c}, {\it  On representations and topological aspects of positive
		maps on non-unital quasi *- algebras},   Positivity {\bf 28}   (2024) art. n.  66 (29 pp.).	
\bibitem{BKMS} A.M. Bikchentaev, F. Kittaneh , M.S Moslehian , Y. Seo, Trace Inequalities:
For Matrices and Hilbert Space Operators, {\em Forum for Interdisciplinary Mathematics (FFIM)}, Springer, 2024.
        \bibitem{bhatt}S.J. Bhatt,  {\it Stinespring representability and Kadison’s Schwarz inequality in non-unital Banach star algebras and applications}. Proc. Indian Acad. Sci. (Math. Sci.) {\bf 108} (1998) 283--303.
\bibitem{DKM} A. Dadkhah, M. Kian and M.S. Moslehian, {\em Decomposition of tracial positive maps and applications in quantum information}, { Analysis and Mathematical Physics } {\bf 14}, art. n. 48, (2024).

        
	\bibitem{FT_book} M. Fragoulopoulou, C. Trapani, {\em Locally Convex Quasi *-algebras and their Representations}, Lecture Notes in Mathematics 2257, Springer, 2020.

    \bibitem{GMX} M. Ghaemi, M.S. Moslehian, and Q. Xu, {\em Kolmogorov decomposition of conditionally completely positive definite kernels}. Positivity {\bf 25} (2021) 515--530.
    \bibitem{heo}J. Heo, {\em Completely -multi-positive linear maps and representations on Hilbert $C^*$-modules}, J. Operator Th. {\bf 41} (1999) 3-22.
     \bibitem{Jocic} D.R. Joci\'{c},
	{\em	Cauchy–Schwarz norm inequalities for weak*-integrals of operator valued functions},
	Journal of Functional Analysis,
	{\bf 218}, 318-346  (2005).
    \bibitem{Joita} M. Joi\c{t}a, {\em Comparison of completely positive maps on Hilbert $C^*$-modules},
J. Math. Anal.  Appl.
{\bf  393} (2012) 644-650.

\bibitem{KK} K. Karimi, K. Sharifi, {\em Completely Positive Maps on Hilbert Modules over Pro-$C^*$-Algebras}, Bulletin Math\'ematique de La Soci\'et\'e Des Sciences Math\'ematiques de Roumanie {\bf 60}(108),  (2017) 181-193.
	

   \bibitem{KMN} M. Kian, M.S. Moslehian and R. Nakamoto,   {\it Asymmetric Choi–Davis inequalities}. Linear and Multilinear Algebra, {\bf 70} (2020) 3287–3300. https://doi.org/10.1080/03081087.2020.1836115
	\bibitem{Lance} E.C. Lance, {\em Hilbert $C^*$-modules: A toolkit for operator algebraists}, London Mathematical Society Lecture Note Series vol. 210, Cambridge University Press (1995).
	 \bibitem{MT} V. Manuilov and E.V. Troitsky, {\em Hilbert $C^*$-modules}, AMS Translations Math. Monographs vol. 226, (2005).
	
	
    \bibitem{MJJ} M.S. Moslehian, M. Joita, and U.C. Ji, \textit{ KSGNS Type Constructions for 
-Completely Positive Maps on Krein 
-Modules}. Complex Anal. Oper. Theory {\bf 10}(2016) 617--638.
	 \bibitem{paulsen} V. Paulsen, {\em Completely Bounded Maps and Operator Algebras}, 
	 Cambridge University Press, 2003.

      \bibitem{PY} J-P. Pellonp\"a\"a, K. Ylinen, {\em Modules, completely positive maps, and a generalized KSGNS construction}, Positivity {\bf 15} (2011) 509-525. 
	\bibitem{schm_book}{ K. Schm\"udgen}, {\it Unbounded operator algebras
		and Representation theory}, Birkh{\"a}user, Basel, 1990.
	  \bibitem{stinespring} W.F. Stinespring,  {\it Positive functions on $C^*$-algebras}, Proc. Amer. Math. Soc.  {\bf 6} (1955), 211-216.
 \bibitem{stormer} E. St{\o}rmer, {\it Positive linear maps on Operator algebras}, Springer Monographs in Mathematics, Heidelberg, 2013.
\end{thebibliography}

\end{document}